\documentclass[11pt,oneside,a4paper]{amsart}

\usepackage{amscd}
\usepackage{amssymb}
\usepackage{amstext}
\usepackage{amsthm}
\usepackage{mathrsfs}

\usepackage{latexsym}
\usepackage{xcolor}
\numberwithin{equation}{section}

\DeclareMathOperator{\Hol}{Hol}
\DeclareMathOperator{\dist}{dist}

\newcommand{\CC}{\mathbb{C}}

\renewcommand{\phi}{\varphi}

\newcommand{\kl}{{\Bbbk_\lambda}}

\newtheorem{Thm}{Theorem}[section]
\newtheorem{theorem}[Thm]{Theorem}
\newtheorem{lemma}[Thm]{Lemma}
\newtheorem{proposition}[Thm]{Proposition}

\begin{document}
\sloppy
\title[Riesz bases
of reproducing kernels]
{Riesz bases
of reproducing kernels in Fock type spaces and de Branges spaces}

\author{Anton Baranov, Yurii Belov, Alexander Borichev}
\address{Anton Baranov,
\newline Department of Mathematics and Mechanics, St.~Petersburg State University, St.~Petersburg, Russia,
\newline National Research University  Higher School of Economics, St.~Petersburg, Russia,
\newline {\tt anton.d.baranov@gmail.com}
\smallskip
\newline \phantom{x}\,\, Yurii Belov,
\newline Chebyshev Laboratory, St.~Petersburg State University, St. Petersburg, Russia,
\newline {\tt j\_b\_juri\_belov@mail.ru}
\smallskip
\newline \phantom{x}\,\, Alexander Borichev,
\newline I2M, CNRS, Centrale Marseille, Aix-Marseille Universit\'e, 13453 Marseille, France, 
\newline  Department of Mathematics and Mechanics, Saint Petersburg State University, St.~Petersburg, Russia,
\newline {\tt alexander.borichev@math.cnrs.fr}
}
\thanks{The work was supported by Russian Science Foundation grant 14-41-00010.}

\begin{abstract} We describe the radial Fock type spaces which possess Riesz bases of normalized reproducing kernels and which are (are not) isomorphic to 
de Branges spaces in terms of the weight functions.
\end{abstract}

\maketitle

\section{Introduction and main results}

Let $\mathcal F$ be a Hilbert space of entire functions. Assume that (i) 
$\mathcal F$ possesses  the division property,
that is, if $f\in\mathcal F$, $f(\lambda)=0$, then $f/(\cdot-\lambda)\in \mathcal F$, and (ii)  
$\mathcal F$ possesses  the bounded point evaluation property,
that is, for each $\lambda \in \CC$,  the mapping $L_\lambda: f \mapsto f(\lambda)$  is a bounded linear functional on $\mathcal F$.  For every 
$\lambda\in \CC$ there exists
 ${\bf k}_\lambda\in \mathcal F$, the
reproducing kernel at $\lambda$ in $\mathcal F$:
\[
f(\lambda)=\langle f, {\bf k}_\lambda \rangle_{\mathcal F}, \qquad f\in\mathcal F.
\]

Let $\kl={\bf k}_\lambda/\|{\bf k}_\lambda\|$
be the normalized reproducing kernel at $\lambda$. 
Given a sequence 
$\Lambda \subset\mathbb C$,
we say that $\{\Bbbk_{\lambda}\}_{\lambda \in \Lambda}$ is 
a Riesz basis (of normalized reproducing kernels) in $\mathcal F$ if it is complete and for some $c,C>0$ we have
$$
c\sum_{\lambda \in \Lambda}|a_\lambda|^2\le \Bigl\|\sum_{\lambda \in \Lambda}
a_\lambda \Bbbk_{\lambda}\Bigr\|^2
\le C\sum_{\lambda \in \Lambda}|a_\lambda|^2,  
$$
for finite sequences $\{a_\lambda\}$ of complex numbers. 
Equivalently, $\{\Bbbk_{\lambda}\}_{\lambda\in \Lambda}$
is a linear isomorphic image
of an orthonormal basis in a separable Hilbert space. 

The de Branges spaces $\mathcal H(\mathcal E)$ are Hilbert spaces of entire functions determined by Hermite--Biehler class entire functions $\mathcal E$, see \cite{br}. The norm in these spaces is 
given by $\|F\|^2_{\mathcal H(\mathcal E)}=\int_{\mathbb R}|F(x)|^2/|\mathcal E(x)|^2\,dx$.
They possess Riesz bases of normalized reproducing kernels at real points $\{t_n\}_{n\in N}$. 
Correspondingly, every space $\mathcal H(\mathcal E)$ can be identified with
the space $\mathcal{H}$ of all entire functions of the form
$$
F(z) = A(z) \sum_{n\in N} \frac{a_n \mu_n^{1/2}}{z-t_n}, \qquad \{a_n\}\in\ell^2,
$$
where $\{t_n\}_{n\in N}$ is an increasing sequence such that $|t_n|\to\infty$,
$|n| \to\infty$, $N=\mathbb Z$, $\mathbb Z_+$ or $\mathbb Z_-$, 
$\sum_{n\in N} \mu_n \delta_{t_n}$ is a positive measure on $\mathbb{R}$
satisfying $\sum_{n\in N} \frac{\mu_n}{t_n^2+1}<\infty$, 
$A$ is an entire function with zero set $\{t_n\}_{n\in N}$ which is real on the real line, and 
the norm of $F$ is defined
as $\|F\|_{\mathcal{H}} = \|\{a_n\}\|_{\ell^2}$.
On the other hand, every $\mathcal F$ as above possessing 
a Riesz basis of normalized reproducing kernels at real points is (up to the norm equivalence) a de Branges space; see, for instance, \cite{BMS}.

Given a continuous function (a weight) $h$ defined on $[0,+\infty)$,
we extend it to the whole complex plane $\CC$ by $h(z)=h(|z|)$, and consider
the Fock type space 
$$
\mathcal F_h=\{f\in\Hol(\CC):\|f\|^2=\|f\|_h^2= \int_{\CC}|f(z)|^2e^{-h(z)}\,dm(z)<\infty\},
$$
where $dm$ is area Lebesgue measure. This is a Hilbert space of entire functions satisfying properties (i) and (ii).

Seip proved in \cite{SW1} that the standard Fock space $\mathcal F_h$, $h(z)=|z|^2$, has no Riesz basis of normalized reproducing kernels.
Later on, it was proved that $\mathcal F_h$ has no Riesz basis of normalized reproducing kernels for sufficiently regular $h$ such that $h(t)\gg(\log t)^2$ \cite{BDK,bl}. 
On the other hand, the spaces $\mathcal F_h$ with $h(t)=(\log t)^\alpha$, $1<\alpha\le 2$, possess Riesz bases of normalized reproducing kernels at real points \cite{bl}, 
and hence, are de Branges spaces. 

In the opposite direction, Theorem~1.2 in \cite{bbb} states that the de Branges space $\mathcal H$ determined by the spectral data 
$(\{t_n\}_{n\in N}, \{\mu_n\}_{n\in N})$ coincides (up to the norm equivalence) with a Fock type space $\mathcal F_h$ if and only if 
(i) the sequence $\{t_n\}_{n\in N}$ is lacunary: $\liminf_{t_n \to \infty} \frac{t_{n+1}}{t_n}>1$, $\liminf_{t_n \to -\infty} \frac{|t_{n}|}{|t_{n+1}|}>1$, and (ii) 
for some $C>0$ and any $n$,
\begin{equation*}
\sum_{|t_k|\leq|t_n|}\mu_k+t^2_n\sum_{|t_k|>|t_n|}\frac{\mu_k}{t^2_k}\le C \mu_n.
\end{equation*}
(These de Branges spaces $\mathcal H$ are characterized by the property that every complete and  minimal system of reproducing kernels is a strong Markushevich basis (strong $M$-basis), see \cite{bbb}.)

The main question we deal with in this paper is to find conditions on $h$ determining whether the Fock type space $\mathcal F_h$ 
has a Riesz basis of (normalized) reproducing kernels and whether $\mathcal F_h$ 
is a de Branges space, or, correspondingly, whether the norm in $\mathcal F_h$ 
is equivalent to a weighted $L^2$ norm along the real line. Our work can be considered as 
completing a series of works  \cite{bl}, \cite{BMS}, and \cite{bbb}. 

From now on let us fix $\psi(t)=h(e^t)$.

\begin{theorem}\label{t1} If $\psi''$ increases to $+\infty$, then $\mathcal F_h$ has no Riesz bases 
of (normalized) reproducing kernels.
\end{theorem}

In this case, $(\log t)^2=o(h(t))$, $t\to\infty$. Under different (in some cases more restrictive) regularity conditions, this result follows from \cite{SW1}, \cite{BDK}, and \cite{bl}.

\begin{theorem}\label{t2} If $\lim_{t\to\infty}\psi'(t)=\infty$, $\psi''$ is a non-increasing positive function, and $|\psi'''(t)|=O(\psi''(t)^{5/3})$, $t\to\infty$, then $\mathcal F_h$ 
has a Riesz basis of (normalized) reproducing kernels at real points, and hence, is (up to the norm equivalence) a de Branges space.
\end{theorem}

\begin{theorem}\label{t3} Given any convex positive function $\phi$ with $x=o(\phi(x))$, $x\to\infty$, there exists a function $\psi$ increasing to $+\infty$ and satisfying 
the property $\psi(t)=o(\phi(t))$, 
$t\to\infty$, such that $\mathcal F_h$ has a Riesz basis   
of (normalized) reproducing kernels and is not (up to the norm equivalence) a de Branges space.
\end{theorem}

\subsection{}\label{ss11} Let $\{{\bf k}_\lambda\}_{\lambda\in\Lambda}$ be a complete minimal family of the reproducing kernels in $\mathcal F_h$. Then there exists a so called 
generating function $E$ with simple zeros at the points of $\Lambda$ such that $E\not\in\mathcal F_h$ and $E_\lambda=E/(\cdot-\lambda)\in\mathcal F_h$ 
for every $\lambda\in\Lambda$. The family $\{\|{\bf k}_\lambda\|E_\lambda/E'(\lambda)\}_{\lambda\in\Lambda}$ is biorthogonal to $\{\kl\}_{\lambda\in\Lambda}$. 
Consider the mappings $T_\Lambda:f\mapsto f|\Lambda$ and 
$$
R_\Lambda:\{a_\lambda\}_{\lambda\in\Lambda} \mapsto E(z)\sum_{\lambda\in\Lambda} \frac{a_\lambda}{E'(\lambda)} \cdot \frac{1}{z-\lambda}.
$$
for finite sequences $\{a_\lambda\}_{\lambda\in\Lambda}$. Then $T_\Lambda R_\Lambda=Id$. If $T_\Lambda$ is bounded from $\mathcal F_h$ to 
$\ell^2(1/\|{\bf k}_\lambda\|^2)$ and $R_\Lambda$ is bounded from $\ell^2(1/\|{\bf k}_\lambda\|^2)$ to $\mathcal F_h$, then the family 
$\{\|{\bf k}_\lambda\|E_\lambda/E'(\lambda)\}_{\lambda\in\Lambda}$ is a Riesz basis, and hence $\{\kl\}_{\lambda\in\Lambda}$ is a Riesz basis of normalized 
reproducing kernels in $\mathcal F_h$ (see, for instance, \cite[Section C.3.1]{nik}).

\subsection{} In this paper, we deal mainly with non-decreasing weights $h$. This does restrict the class of Fock type spaces we consider.

\begin{proposition}\label{ex}
There exist continuous $h_0$ such that $\mathcal F_{h_0}$ does not coincide 
(up to the norm equivalence) with $\mathcal F_h$ for any non-decreasing weight $h$. 
\end{proposition}

Let us introduce some notation.  If $\psi$ is convex, then $h$ is subharmonic and we set 
$\rho(z)=(\Delta h(z))^{-1/2}$, $\tau(t)=(\psi'')^{-1/2}(t)=\rho(e^{t})e^{-t}$. 
Later on, we'll see that $\rho$ will be a local metric coefficient and $\tau$ will be 
a local metric coefficient in the logarithmic scale. 
Given an analytic function $f$, we denote by $Z(f)$ its zero set.

The proofs of Theorems~\ref{t1}--\ref{t3} are contained in Sections~\ref{pt1}--\ref{pt3}, correspondingly. Proposition~\ref{ex} is proved in 
Section~\ref{pex}.

\section{Proof of Theorem~\ref{t1}}
\label{pt1}

We follow the method proposed in \cite{SW1} and later used in \cite{bl}. It turns out that in some large annuli we can obtain precise estimates on the norm of the reproducing kernel. If $\{\kl\}_{\lambda\in\Lambda}$ is a Riesz basis, then $\Lambda$ is $\rho$-separated and $\rho$-dense on these annuli 
(see Lemma~\ref{l6} below) and this implies that the corresponding Hilbert transform is unbounded which leads to a contradiction.

We suppose in this section that $\psi''$ increases to $+\infty$, $\tau$ decreases to $0$.

\begin{lemma}
\label{l1}
Given $A<\infty$, $y_0>0$, there exists $y\ge y_0$ such that
\begin{equation}
\label{e1}
\Bigl|\frac{\psi''(x)}{\psi''(y)}-1\Bigr|\le \frac 1A,
\end{equation}
if $|x-y|\le A\tau(y)$.
\end{lemma}

For similar statements see, for instance, \cite{ha}.

\begin{proof}
Without loss of generality, $A$ is sufficiently large. Set
$$
y_n=y_{n-1}+4A\tau(y_{n-1}),\qquad n\ge 1,
$$
and suppose that 
$$
\psi''(y_n)\ge \frac{A+1}{A}\psi''(y_{n-1}),\qquad n\ge 1.
$$
Then
$$
\psi''(y_n)\ge \Bigl(\frac{A+1}{A}\Bigr)^n\psi''(y_0),\qquad n\ge 1,
$$
and, hence, 
\begin{gather*}
\sum_{n\ge 1}(y_n-y_{n-1})<\infty,\\
\lim_{n\to\infty}y_n=y<\infty,
\end{gather*}
which is absurd. 

Hence, there exists $z\ge y_0$ such that 
$$
\psi''(z)\le \psi''(x)< \frac{A+1}{A}\psi''(z),\qquad 0\le x-z\le 4A\tau(z).
$$
It remains to set $y=z+2A\tau(z)$. For large $A$ we obtain \eqref{e1}.
\end{proof}

\begin{lemma}
\label{l2}
Let $B>0$, $N<\infty$. 
Given $e^y=R_1<R_2<\ldots < R_N$, $R_{N+1}=+\infty$, and integers $k_1,\ldots,k_N$ 
such that
\begin{gather*}
\frac 1B\le \frac{R_{j+1}-R_j}{e^y\tau(y)}\le B,\qquad 1\le j<N,\\
\frac 1B\le k_j\tau(y)\le B,\qquad 1\le j\le N,
\end{gather*}
we set
$$
P(z)=\prod_{1\le j\le N}\Bigl(1-\Bigl(\frac{z}{R_j}\Bigr)^{k_j}\Bigr).
$$
Next, we set
$$
v(t)=\begin{cases}
0,\qquad t<R_1,\\
\sum_{1\le s\le j}k_s(\log t-\log R_s),\qquad R_j\le t< R_{j+1},\, 1\le j\le N.
\end{cases}
$$
Then for every $\delta>0$ for some $D=D(\delta,B)$ independent of $N$ and for all $y>y(\delta,\psi,B)$ 
we have 
$$
\bigl|\log|P(z)|-v(|z|)\bigr|\le D(\delta,B) 
$$
if $\dist(z,Z(P))>\delta e^y\tau(y)$.
\end{lemma}

\begin{proof} It suffices to use that
\begin{align*}
&\bigl|\log|1-z^k|\bigr|\le Ce^{-t}, \qquad |z|=1-\frac tk,\,\, t\ge 1,\\
&\bigl|\log|1-z^k|\bigr|\le D_0(\delta), \qquad 1-\frac1k<|z|<1+\frac1k,\quad \dist(z,\Lambda_k)>\frac\delta{k},\\
&\bigl|\log|1-z^k|-k\log|z|\bigr|\le Ce^{-c\min(t,k)}, \qquad |z|=1+\frac tk,\,\, t\ge 1,
\end{align*}
where $\Lambda_k=\{e^{2\pi ij/k}\}_{0\le j<k}$.
\end{proof}

Given $A,R>0$, set 
$$
\Omega_{R,A}=\{z\in\CC:\bigl||z|-R\bigr|\le A\rho(R)\}.
$$

\begin{lemma}
\label{l3}
Given $A<\infty$, there exist $R>0$ and a polynomial $Q$ such that $Z(Q)\setminus\{0\}\subset \Omega_{R,2A}$ and 
\begin{gather*}
\frac 1M\le \frac{|Q(z)|\rho(R)e^{-h(z)}}{\dist(z,Z(Q))}\le M,\qquad z\in \Omega_{R,A},\\
|Q(z)|\le Me^{h(z)},\qquad z\in\CC,\\
\dist(w,Z(Q)\setminus\{w\})\ge \frac{\rho(R)}{M},\qquad w\in Z(Q)\setminus\{0\},\\
\dist(z,Z(Q)\setminus\{w\})\le M\rho(R),\qquad z\in \Omega_{R,A},
\end{gather*}
for some absolute constant $M$.
\end{lemma}

\begin{proof} The lemma follows immediately from Lemmas~\ref{l1} and \ref{l2}. 
Indeed, by Lemma~\ref{l1} we can find $R$ such that $\psi''$ is almost constant on the corresponding interval around $\log R$. Then we find $k_j,R_j$ in Lemma~\ref{l2} and $a\in\mathbb R$, $k\in\mathbb N$ such that $a+kt+v(\exp t)$ is close to $\psi(t)$ 
for $\exp t\in \Omega_{R,2A}$. Finally, we set $Q(z)=e^az^kP(z)$, where $P$ is given by Lemma~\ref{l2}.
\end{proof}

For a similar construction see \cite{BDK}; for an alternative way to produce analytic functions $Q$ satisfying such conditions see the atomization procedure in \cite{LM}.\\

Now we can proceed like in \cite{bl}.

\begin{lemma} \label{l4}
In the conditions of Lemma~\ref{l3}, given $w\in\Omega_{R,A/2}$, there exists a function $\Phi_w$ analytic
in the disc $D_w=\{z\in\mathbb C:|z-w|<\rho(R)\}$ and such that
$$
\frac 1M\le |\Phi_w(z)|e^{-h(z)}\le M,\qquad z\in D_w,
$$
for some absolute constant $M$.
\end{lemma}

\begin{proof} The proof of this folklore result is identical to that of Lemma~2.1 in \cite{bl}.
\end{proof}

\begin{lemma} \label{l5} 
In the conditions of Lemma~\ref{l3}, 
$$
\frac 1M\le \|{\bf k }_z\| e^{-h(z)/2}\rho(R)\le M, 
\qquad z \in \Omega_{R,A/3},
$$
for some absolute constant $M$.
\end{lemma}

\begin{proof} Analogous to that of Lemma~2.3 in \cite{bl} with $F$ there replaced by our polynomial $Q$.
\end{proof}

\begin{lemma} \label{l6} In the conditions of Lemma~\ref{l3}, 
let a sequence  $\Lambda\subset \CC$ be such that 
$\{\kl \}_{\lambda \in \Lambda}$  is a Riesz basis in $\mathcal F_h$.
Then   
\begin{itemize}
\item[(a)] $\dist(\lambda,\Lambda\setminus
\{\lambda\})\ge \beta\rho(R), \quad \lambda \in \Lambda\cap \Omega_{R,A/4}$,
\item[(b)] $\dist(z,\Lambda)\le \rho(R)/\beta, \quad z\in\Omega_{R,A/4}$,
\end{itemize}
for some $\beta=\beta(\Lambda)$ and for $A>A(\Lambda)$.
\end{lemma}

\begin{proof} Analogous to that of Lemma~2.4 in \cite{bl}. 
\end{proof}

\begin{proof}[Proof of Theorem~\ref{t1}] We follow the proof of Theorem~2.5 in \cite{bl}. 
Suppose that the system 
$\{{\Bbbk}_\lambda \}_{\lambda \in \Lambda}$  is  a Riesz basis
in $\mathcal F_h$. Fix large $A$, choose $R$ and $Q$ as in Lemma~\ref{l3}, and define $E$ as in Subsection~\ref{ss11}. 
Lemma~\rm\ref{l5} implies  
that
$$
\Bigl\| \frac{E}{\cdot-\lambda} \Bigr\|^2
\asymp\frac{|E'(\lambda)|^2}{\|{\bf k }_z\|^2}
\asymp|E'(\lambda)|^2\rho^2(\lambda)e^{-h(\lambda)}, \qquad \lambda \in \Lambda\cap \Omega_{R,A/4}. 
$$
Consider the function $E/[(\cdot-\lambda)\Phi_\lambda]$. Applying  the mean value property 
we obtain that 
$$
\int_{\CC}\frac{|E(z)|^2}{|z-\lambda|^2}e^{-h(z)}dm(z)
\le \frac C{\rho^2(R)}
\int_{|z-\lambda|<\rho(R)}
|E(z)|^2e^{-h(z)}dm(z).
$$
By Lemma~\ref{l6}~(a), we have
\begin{gather*}
\int_{\CC}\Bigl(\sum_{|\lambda-R|<N\rho(R), \, \lambda \in \Lambda}
\frac{1}{|z-\lambda|^2}\Bigr)|E(z)|^2e^{-h(z)}dm(z)\\
\le \frac C{\rho^2(R)}\int_{|R-z|<(N+1)\rho(R)}|E(z)|^2e^{-h(z)}dm(z),
\end{gather*} 
with $A/4-1<N\le A/4$, and, as a result,
\begin{equation}
\inf_{z:|z-R|<(N+1)\rho(R)}
\Bigl[\rho^{2}(R)\Bigl(\sum_
{|\lambda-R|< N\rho(R),\, \lambda \in \Lambda}
\frac{1}{|z-\lambda|^2}\Bigr)\Bigr]\le C.
\label{06}
\end{equation}

Finally, by Lemma~\ref{l6}~(b),
$$
\rho^{2}(R)\Bigl(\sum_{|\lambda-R|< N\rho(R),\, \lambda \in \Lambda}
\frac{1}{|z-\lambda|^2}\Bigr)\ge 
     C\int_{\rho(R)<|\zeta|<N\rho(R)}
  \frac{dm(\zeta)}{|z-\zeta|^2}.
$$
For large $A$ (and hence, large $N$) we get a contradiction to \eqref{06}.
\end{proof}

\section{Proof of Theorem~\ref{t2}}
\label{pt2}

The situation for $h$ of slow growth is quite different from that considered in Section~\ref{pt1}. In particular, along some special (lacunary) sequence of points 
$\{\lambda_n\}_{n\ge 0}$ we obtain here that $\|{\bf k}_{\lambda_n}\|^2\asymp e^{h(\lambda_n)}/(\lambda_n\rho(\lambda_n))\gtrsim 
e^{h(\lambda_n)}/(\rho(\lambda_n))^2$ instead of $\|{\bf k}_{\lambda_n}\|^2\asymp   
e^{h(\lambda_n)}/(\rho(\lambda_n))^2$ in Lemma~\ref{l5}. To prove the boundedness of the operator $R_\Lambda$ introduced in Subsection~\ref{ss11} we use the results of \cite{BMS}.

For $n\ge \psi'(0)/2-1$ we choose $y_n$ such that $\psi'(y_n)=2n+2$ and set $\lambda_n=\exp y_n$. Define 
$$
g(t)=\psi(y_n+t)-(y_n+t)\psi'(y_n),\qquad -y_n\le t<\infty. 
$$
Set $\alpha_n=\psi''(y_n)=g''(0)$. Then $g'(0)=0$, $0<g''(t)\le \psi''(0)$, $|g'''(t)|=O(g''(t)^{5/3})$, and hence,
\begin{align}
g(t)-g(0)&\asymp \alpha_nt^2,\qquad |t|\le C\alpha_n^{-2/3}\label{e31},\\
g'(t)&\asymp \alpha_nt,\qquad |t|\le C\alpha_n^{-2/3}\label{e3star},\\
g(t)-g(0)&\ge C\alpha_n^{-1/3}+C\alpha_n^{1/3}|t|,\qquad |t|\ge C\alpha_n^{-2/3}\label{e32}.
\end{align}
We assign to the finite number of $y_n$ not yet defined, arbitrary values in such a way that the sequence $\{y_n\}_{n\ge 0}$ 
of positive numbers is strictly increasing.

\begin{lemma}
\label{l1q}
$$
\|z^n\|^2\asymp e^{(2n+2)y_n-\psi(y_n)}(\psi''(y_n))^{-1/2},\qquad n\ge 0.
$$
\end{lemma}

\begin{proof} We have
$$
\|z^n\|^2=O(1)+2\pi\int_1^\infty t^{2n+1}e^{-\psi(\log t)}\, dt=O(1)+2\pi\int_0^\infty e^{(2n+2)s-\psi(s)}\, ds.
$$
By \eqref{e31} and \eqref{e32} we have 
$$
\frac{C_1}{\sqrt{\alpha_n}}\le \int_{-y_n}^\infty e^{g(0)-g(t)}\, dt\le \frac{C_2}{\sqrt{\alpha_n}}
$$
and hence, 
$$
\frac{C_1}{\sqrt{\alpha_n}}\le  \|z^n\|^2 e^{\psi(y_n)-(2n+2)y_n}\le   \frac{C_2}{\sqrt{\alpha_n}}.
$$
\end{proof}

Next, we estimate the norm of the reproducing kernel at the points $\lambda_n$; compare also to 
Lemma~\ref{l5}. 

\begin{lemma}
\label{l2q}
$$
\|{\bf k}_{\lambda_n}\|^2\asymp e^{\psi(y_n)-2y_n}(\psi''(y_n))^{1/2}=e^{h(\lambda_n)}/(\lambda_n\rho(\lambda_n)),\qquad n\ge 0.
$$
\end{lemma}

\begin{proof} First of all, by Lemma~\ref{l1q}, we have
$$
\|{\bf k}_{\lambda_n}\|^2\ge \frac{\lambda_n^{2n}}{\|z^n\|^2}\ge C e^{\psi(y_n)-2y_n}(\psi''(y_n))^{1/2}.
$$

Let $f\in\mathcal F_h$,
\begin{gather*}
F(t)=\Bigl(\int_0^{2\pi}|f(te^{i\theta}|^2\,d\theta\Bigr)^{1/2},\\
\omega(s)=\log F(\exp s).
\end{gather*}
Hardy's convexity theorem \cite[Chapter 1]{Du} states that $\omega$ is convex. 

Next, by \eqref{e3star}, for large $n$ we have $y_n-\alpha_n^{-1/2}>y_{n-1}$ and 
\begin{gather*}
\|f\|^2=\int_0^\infty F(t)^2e^{-h(t)}t\,dt\ge \int_0^\infty e^{2\omega(s)-\psi(s)+2s}\,ds\\ \ge 
\int_{y_n-\alpha_n^{-1/2}}^{y_n+\alpha_n^{-1/2}}e^{2\omega(s)-\psi(s)+2s}\,ds\\ \ge 
Ce^{-g(0)}\int_{-\alpha_n^{-1/2}}^{\alpha_n^{-1/2}}e^{2\omega(y_n+s)-2n(y_n+s)}\,ds=I. 
\end{gather*}

If $\delta=\alpha_0^{-1/2}$, then by convexity of $\omega$ we have 
\begin{gather*}
I\ge Ce^{-g(0)}\alpha_n^{-1/2}\int_{-\delta}^{\delta}e^{2\omega(y_n+s)-2n(y_n+s)}\,ds\\ =
Ce^{-g(0)}\alpha_n^{-1/2}\int_{y_n-\delta}^{y_n+\delta}e^{2\omega(s)-2ns}\,ds\\ =
Ce^{-g(0)}\alpha_n^{-1/2}\int_{\exp(y_n-\delta)}^{\exp(y_n+\delta)}\frac{|f(z)|^2}{|z|^{2n+2}}\,dm(z).
\end{gather*}

Applying the mean value theorem to the function $f(z)z^{-n-1}$, we conclude that 
$$
|f(e^{y_n})|^2\le C\|f\|^2e^{g(0)}\alpha_n^{1/2}e^{2ny_n}=C\|f\|^2e^{\psi(y_n)-2y_n}\alpha_n^{1/2}.
$$ 
\end{proof}

Next, we consider a continuous piecewise linear function $\ell$ such that 
$$
\ell'(t)=2n+2,\qquad y_n<t<y_{n+1},\, n\ge 0.
$$
Then $\ell'(t)\le \psi'(t)$ for large $t$, and hence, $\ell(t)\le \psi(t)+O(1)$, $t\to\infty$. Next, $\ell'(t)+2\ge \psi'(t)$ for large $t$, and hence, 
$\ell(t)+2t\ge \psi(t)+O(1)$, $t\to\infty$.

\begin{lemma}
\label{l3q}
Given a small $\delta>0$ we have 
\begin{align*}
{\rm (a)}&\quad \int_0^\infty e^{\ell(t)-\psi(t)}\, dt<\infty,\\ 
{\rm (b)}&\quad \int_{y_n-\delta}^\infty e^{\ell(t)-\psi(t)}\, dt \le Ce^{\ell(y_n)-\psi(y_n)}(\psi''(y_n))^{-1/2},\\
{\rm (c)}&\quad \int_0^{y_{n}-\delta} e^{\ell(t)-\psi(t)+2t}\, dt \le Ce^{\ell(y_n)-\psi(y_n)+2y_n}(\psi''(y_n))^{-1/2},\\
{\rm (d)}&\quad \sum_{s=0}^n e^{\psi(y_s)-\ell(y_s)}(\psi''(y_s))^{1/2} \asymp e^{\psi(y_n)-\ell(y_n)}(\psi''(y_n))^{1/2},\\
{\rm (e)}&\quad \sum_{s=n}^\infty e^{\psi(y_s)-\ell(y_s)-2y_s}(\psi''(y_s))^{1/2} \asymp e^{\psi(y_n)-\ell(y_n)-2y_n}(\psi''(y_n))^{1/2}.
\end{align*}
\end{lemma}

\begin{proof} Set $u=\psi-\ell$. Then
\begin{align*}
u'(y_n+t)&\ge C\alpha_n t, \qquad 0\le t\le C\alpha_n^{-2/3},\\
u'(y_n+t)&\ge C\alpha_n^{1/3}, \qquad C\alpha_n^{-2/3}<t<y_{n+1}-y_n,\\
u'(y_n-t)&\ge 1, \qquad 0\le t\le C\alpha_n^{-2/3}.
\end{align*} 
Therefore,
$$
u(y_n)\ge u(y_{n-1})+C\alpha_n^{-2/3}.
$$
Since $\psi''$ does not increase, (a), (b), and (d) follow immediately.

Next, set $w(t)=\ell(t)+2t-\psi(t)$. Then
\begin{align*}
w'(y_n+t)&\ge 1, \qquad 0\le t\le C\alpha_n^{-2/3},\\
w'(y_n-t)&\ge C\alpha_n t, \qquad 0\le t\le C\alpha_n^{-2/3},\\
w'(y_n-t)&\ge C\alpha_n^{1/3}, \qquad C\alpha_n^{-2/3}<t<y_{n}-y_{n-1}.
\end{align*} 
Therefore,
$$
w(y_{n+1})\ge w(y_n)+C\alpha_n^{-2/3}.
$$
Since $\psi''$ does not increase, (c) and (e) follow immediately.
\end{proof}

Set 
$$
E(z)=\prod_{n\ge 0}\Bigl(1-\frac{z}{\lambda_n}\Bigr).
$$
Arguing like in the proof of Lemma~\ref{l2}, for small $\delta>0$ we obtain 
\begin{align*}
|E(e^t)|^2&\le Ce^{\ell(t)},\qquad t\ge 0,\\
|E(e^t)|^2&\asymp e^{\ell(t)},\qquad \dist(t,\{y_n\}_{n\ge 0})\ge \delta.
\end{align*}
Therefore, for every entire function $F\not\equiv 0$ we have 
\begin{equation}
FE\not\in\mathcal F_h.
\label{es1}
\end{equation}
Indeed,
$$
\int_{\mathbb C}|F(z)|^2|E(z)|^2e^{-h(z)}\,dm(z)\ge 
C\int_0^\infty e^{\ell(t)-\psi(t)+2t}\, dt=\infty.
$$
In a similar way, by Lemma~\ref{l3q}~(a), $E/(\cdot-\lambda_0)\in \mathcal F_h$.
Furthermore,
\begin{equation}
|E'(\lambda_n)|^2\asymp e^{\ell(y_n)-2y_n}.
\label{es2}
\end{equation}

Set $d\mu(z)=|E(z)|^2e^{-h(z)}\,dm(z)$, $v_n=e^{\psi(y_n)-\ell(y_n)}(\psi''(y_n))^{1/2}$, $\Omega_n=\{z:e^{y_n-\delta}<|z|<e^{y_{n+1}-\delta}\}$, $n\ge 0$.
By Lemma~\ref{l3q}~(b), we have 
\begin{gather}
\label{et1}\int_{\Omega_n}\frac{d\mu(z)}{|z-\lambda_n|^2}=\int_{e^{y_n-\delta}<|z|<e^{y_{n+1}-\delta}}\Bigl|\frac{E(z)}{z-e^{y_n}}\Bigr|^2e^{-h(z)}\,dm(z)\\ 
\notag\asymp
\int_{y_n-\delta}^{y_{n+1}-\delta} e^{\ell(t)-2(t-y_n)^+-\psi(t)+2(t-y_n)}\, dt\le \int_{y_{n}-\delta}^\infty e^{\ell(t)-\psi(t)}\, dt  \\ \le C\frac{e^{\ell(y_n)-\psi(y_n)}}{(\psi''(y_n))^{1/2}}=\frac{C}{v_n},\qquad n\ge 0.\notag
\end{gather}
Furthermore, again by Lemma~\ref{l3q}~(b), we have 
\begin{gather*}
\sum_{m=n+1}^\infty\int_{\Omega_m}\frac{d\mu(z)}{|z|^2}=\int_{y_{n+1}-\delta}^\infty e^{\ell(t)-\psi(t)}\, dt\le \frac{C}{v_{n+1}},\quad n\ge 0.
\end{gather*}
By Lemma~\ref{l3q}~(d), we have 
\begin{gather*}
\sum_{s=0}^n v_s\asymp v_n,\quad n\ge 0.
\end{gather*}
Therefore, 
\begin{equation}
\sum_{s=0}^n v_s \sum_{m=n+1}^\infty\int_{\Omega_m}\frac{d\mu(z)}{|z|^2}\le C,\quad n\ge 0.
\label{et2}
\end{equation}

Next, by Lemma~\ref{l3q}~(c), we have 
\begin{gather*}
\sum_{m=0}^n\int_{\Omega_m}d\mu(z)=\int_0^{y_{n+1}-\delta} e^{\ell(t)-\psi(t)+2t}\, dt \le C\frac{e^{2y_{n+1}}}{v_{n+1}},\quad n\ge 0.
\end{gather*}
By Lemma~\ref{l3q}~(e), we have 
\begin{gather*}
\sum_{s=n+1}^\infty \frac{v_s}{|\lambda_s|^2}\asymp \frac{v_{n+1}}{|\lambda_{n+1}|^2},\quad n\ge 0.
\end{gather*}
Therefore, 
\begin{equation}
\sum_{s=n+1}^\infty \frac{v_s}{|\lambda_s|^2} \sum_{m=0}^n\int_{\Omega_m}d\mu(z)\le C,\quad n\ge 0.
\label{et3}
\end{equation}

Set $\Lambda=\{\lambda_n\}_{n\ge 0}$ and consider the mappings $T_\Lambda$ and $R_\Lambda$ defined in Subsection~\ref{ss11}. 
The argument in the proof of Lemma~\ref{l2q} implies that the mapping $T_\Lambda$ 
is bounded from $\mathcal F_h$ to $\ell^2(1/\|{\bf k}_\lambda\|^2)$. 
By \eqref{es1}, $T_\Lambda$ is injective. 
By Lemma~\ref{l2q} and by \eqref{es2}, $v_n\asymp \|{\bf k}_{\lambda_n}\|^2/|E'(\lambda_n)|^2$. 
Estimates \eqref{et1}--\eqref{et3} permit us to apply Theorem~1.1 in \cite{BMS} to show that 
$d\mu$ is a Carleson measure for the space of the corresponding discrete Hilbert transforms.  
Thus, 
$R_\Lambda$ is bounded from $\ell^2(1/\|{\bf k}_\lambda\|^2)$ to $\mathcal F_h$. 
By the argument in Subsection~\ref{ss11} 
we conclude that 
$\{\Bbbk_{\lambda_n}\}_{n\ge 0}$ is a Riesz basis in $\mathcal F_h$.

\section{Proof of Theorem~\ref{t3}}
\label{pt3}

In this section we construct an increasing weight $h$ and a corresponding set $\Lambda$ consisting of (an increasing amount of) equidistributed points on rapidly growing concentric circles such that $\{\Bbbk_{\lambda}\}_{\lambda\in\Lambda}$ is a Riesz basis in $\mathcal F_h$.

Given an increasing sequence $\{R_n\}_{n\ge 1}$ we set 
\begin{multline*}
\psi(t)=t+2\sum_{s\le n}s(t-\log R_s)+n\min(t-\log R_n,\log R_{n+1}-t), \\ \log R_n\le t\le \log R_{n+1},\, n\ge 1.
\end{multline*}
Clearly, $\psi$ is an increasing function. 

An elementary geometric argument permits us to choose a sequence $\{R_n\}_{n\ge 1}$ 
satisfying the following two properties:
$$
\text{\rm (I)}\qquad R_1\ge 2,\qquad R_{n+1}\ge R^2_n,\quad n\ge 1
$$
and
$$
\text{\rm (II)}\qquad \psi(t)=o(\phi(t)), \qquad t\to\infty.
$$
In particular, $\log R_n\ge 2^n\log 2$. 
As above, we fix $h(t)=\psi(\log t)$. Next, we define
$$
E(z)=\prod_{n\ge 1}\Bigl(1-\Bigl(\frac{z}{R_n}\Bigr)^n\Bigr),\qquad \Lambda=Z(E).
$$

By an argument similar to that in the proof of Lemma~\ref{l2}, we have
\begin{align}
&|E(z)|\asymp e^{h(z)/2}\frac{n\dist(z,\Lambda)}{R^{3/2}_n},\qquad \bigl||z|-R_n\bigr|\le \frac{R_n}{n},\notag\\
&|E'(\lambda)|\asymp e^{h(R_n)/2} \frac{n}{R^{3/2}_n},\qquad \lambda\in\Lambda,\, |\lambda|=R_n,\label{e14}\\
&|E(z)|\asymp e^{h(z)/2}|z|^{-1/2}\Bigl| \frac{R_n}{z}\Bigr|^{n/2},\qquad \frac{n+1}{n}R_n\le |z|\le \sqrt{R_nR_{n+1}},\notag\\
&|E(z)|\asymp e^{h(z)/2}|z|^{-1/2}\Bigl| \frac{z}{R_{n+1}}\Bigr|^{n/2},\qquad \sqrt{R_nR_{n+1}}\le |z|\le \frac{nR_{n+1}}{n+1}.\notag
\end{align}
Let $F\not=0$ be an entire function. Then
\begin{multline}
\int_{\CC}|F(z)|^2|E(z)|^2e^{-h(z)}\,dm(z)\\ \ge 
C \sum_{n\ge 1}\int_{\frac{n-1}{n}\le \frac{|z|}{R_n}\le \frac{n+1}{n}}|F(z)|^2\frac{n^2\dist^2(z,\Lambda)}{R^3_n}\,dm(z)
\ge \sum_{n\ge 1}\frac{R_n}{n}=\infty.
\label{e01}
\end{multline}
Thus, $E\not\in\mathcal F_h$ and for any entire function $F\not=0$ we have $FE\not\in\mathcal F_h$.

Given $\lambda\in\Lambda$, we set $E_\lambda=\frac{E}{\cdot-\lambda}$. 
Next, we are going to deal with the scalar products $\langle E_\lambda,E_\mu\rangle_{\mathcal F_h}$. 
For $\lambda\in\Lambda$, $|\lambda|=R_n$, we have 
\begin{multline}
\|E_\lambda\|^2=\int_{\CC}\frac{|E(z)|^2}{|z-\lambda|^2}e^{-h(z)}\,dm(z)\\ \le 
C \Bigl(\int_{\frac{n-1}{n}\le \frac{|z|}{R_n}\le \frac{n+1}{n}}\frac{n^2\dist^2(z,\Lambda)}{R^3_n\,|z-\lambda|^2}\,dm(z)
\\+\sum_{s\ge 1,\,s\not=n}\int_{\frac{s-1}{s}\le \frac{|z|}{R_s}\le \frac{s+1}{s}}\frac{s^2\dist^2(z,\Lambda)}{R^3_s\,|z-\lambda|^2}\,dm(z)
\\+\sum_{s\ge 1}\int_{\frac{s+1}{s}R_s\le |z|\le \sqrt{R_sR_{s+1}}}
\Bigl|\frac{R_s}{z}\Bigr|^s\frac{dm(z)}{|z|\cdot|z-\lambda|^2} 
\\+\sum_{s\ge 1}\int_{\sqrt{R_sR_{s+1}}\le |z|\le \frac{s}{s+1}R_{s+1}}
\Bigl|\frac{z}{R_{s+1}}\Bigr|^s\frac{dm(z)}{|z|\cdot|z-\lambda|^2}
\Bigr)
\\ \le \frac{C}{R_n}+C\sum_{s\ge 1,\,s\not=n}\frac{R_s}{s\max(R_s,R_n)^2}\asymp\frac1{R_n}.
\label{e2}
\end{multline}
Furthermore, if $\lambda,\mu\in\Lambda$, $R_j=|\lambda|<|\mu|=R_k$, then in a similar way 
\begin{equation}
\int_{\CC}\frac{|E(z)|^2}{|z-\lambda|\cdot|z-\mu|}e^{-h(z)}\,dm(z)\le \frac{C\log(k+1)}{R_k}.
\label{e3}
\end{equation}

Next we fix $n\ge 1$ and set $R=R_n$. Then 
$$
E(z)=\Bigl(1+O\Bigl(\frac{1}{R}\Bigr)\Bigr)\Bigl(1-\Bigl(\frac{z}{R}\Bigr)^n\Bigr)\prod_{1\le s<n}\Bigl(\frac{z}{R_s}\Bigr)^s,\quad \bigl||z|-R\bigr|\le\frac{R}{2}.
$$
Let $\alpha\not=\beta$ be such that $|\alpha|=|\beta|=1$, and let $\lambda=\alpha R$, $\mu=\beta R$ belong to $\Lambda$. We have 
\begin{gather*}
\int_{\CC}\frac{|E(z)|^2}{(z-\lambda)\overline{(z-\mu)}}e^{-h(z)}\,dm(z)\\=O\Bigl(\frac{1}{nR}\Bigr)+
\int_{R/2\le |z|\le R}\frac{\Bigl(1-\Bigl(\dfrac{z}{R}\Bigr)^n\Bigr)\Bigl(1-\Bigl(\dfrac{\overline{z}}{R}\Bigr)^n\Bigr)}{(z-\lambda)\overline{(z-\mu)}}
\Bigl|\frac{z}{R}\Bigr|^{n-1}\,\frac{dm(z)}{|z|}\\+
\int_{R\le |z|\le 3R/2}\frac{\Bigl(1-\Bigl(\dfrac{z}{R}\Bigr)^n\Bigr)\Bigl(1-\Bigl(\dfrac{\overline{z}}{R}\Bigr)^n\Bigr)}{(z-\lambda)\overline{(z-\mu)}}
\Bigl|\frac{R}{z}\Bigr|^{3n}\,\frac{dm(z)}{|z|}\\
=O\Bigl(\frac{1}{nR}\Bigr)+\frac{1}{R}\int_{1/2}^1 \Bigl(\sum_{s=0}^{n-1}\alpha^s w^{n-1-s}\Bigr) \overline{\Bigl(\sum_{s=0}^{n-1}\beta^s w^{n-1-s}\Bigr)}|w|^{n-2}\,dm(w)\\
+\frac{1}{R}\int_1^{3/2} \Bigl(\sum_{s=0}^{n-1}\alpha^s w^{n-1-s}\Bigr) \overline{\Bigl(\sum_{s=0}^{n-1}\beta^s w^{n-1-s}\Bigr)}|w|^{-3n-1}\,dm(w)\\
=O\Bigl(\frac{1}{nR}\Bigr)+\frac{1}{R} \sum_{s=0}^{n-1} 2\pi (\alpha\overline{\beta})^s
\Bigl(\frac1{2s+n+1}+\frac1{3n-2s-2}\Bigr).
\end{gather*}
Set
\begin{gather*}
a_s=\frac1{2s+n+1}+\frac1{3n-2s-2},\qquad 0\le s<n,\\
b_{j,k}=\sum_{s=0}^{n-1}e^{2\pi i(j-k)s/n}a_s,\qquad 0\le j,k< n,\\
A=\bigl(b_{j,k}\bigr)_{0\le j,k< n}.
\end{gather*}
Then $A$ is a circulant matrix. It is well-known and easily verified that $A$ has an orthonormal basis of eigenvectors $n^{-1/2}(e^{2\pi i kq/n})_{0\le k<n}$, 
$0\le q<n$, and the corresponding eigenvalues are $\sum_{k=0}^{n-1}b_{0,k}e^{2\pi i kq/n}$, $0\le q<n$. Therefore, 
the norm of the corresponding operator acting on $\ell^2_n$ is equal to 
$$
\max_{0\le q<n}\Bigl| \sum_{k=0}^{n-1}e^{2\pi ikq/n}\sum_{s=0}^{n-1}e^{-2\pi iks/n}a_s \Bigr|=n\max_{0\le q<n}|a_q|=O(1),\qquad n\to\infty.
$$
Thus, if 
$$
B_n=\Bigl(\langle E_\lambda,E_\mu\rangle_{\mathcal F_h}\Bigr)_{\lambda,\,\mu\in\Lambda,\,|\lambda|=|\mu|=R_n},  
$$
then 
\begin{equation}
\|B_n\|_{\ell^2_n\to\ell^2_n}=O\Bigl(\frac{1}{R_n}\Bigr),\qquad n\to\infty.
\label{e5}
\end{equation}

Now, for $|\lambda|=R_n$ we estimate $\|{\bf k}_\lambda\|$. Applying the mean value theorem to the function 
$$
g(z)=f(z)\prod_{1\le s<R_n}\Bigl(\frac{R_s}{z}\Bigr)^s
$$
in the disc $D=\{z:|z-\lambda|<R_n/n\}$, we obtain that 
\begin{equation}
|f(\lambda)|^2e^{-h(\lambda)}\le \frac{Cn^2}{R_n^2}\int_D |f(w)|^2e^{-h(w)}\, dm(w),
\label{e7}
\end{equation}
and hence, 
$$
\|{\bf k}_\lambda\|\le Ce^{h(R_n)/2}\frac{n}{R_n},\qquad |\lambda|=R_n.
$$

By \eqref{e14} and \eqref{e2}, for $\lambda\in \Lambda$ such that $|\lambda|=R_n$ we have 
$$
\|{\bf k}_\lambda\|\ge \frac{|E_\lambda(\lambda)|}{\|E_\lambda\|}=\frac{|E'(\lambda)|}{\|E_\lambda\|}\ge Ce^{h(R_n)/2}\frac{n}{R_n}.
$$
Replacing $E(z)$ by $E(e^{i\theta}z)$ we get the same estimate for all $\lambda$ such that $|\lambda|=R_n$. Thus, 
\begin{equation}
\frac1C e^{h(R_n)/2}\frac{n}{R_n}\le \|{\bf k}_\lambda\|\le Ce^{h(R_n)/2}\frac{n}{R_n},\qquad |\lambda|=R_n,
\label{exy}
\end{equation}
for some $C$ independent of $\lambda$ and $n$. 

Now, we consider the mappings $T_\Lambda$ and $R_\Lambda$ defined in Subsection~\ref{ss11}. Estimate 
\eqref{e7} implies that the mapping $T_\Lambda$ 
is bounded from $\mathcal F_h$ to $\ell^2(1/\|{\bf k}_\lambda\|^2)$. 
By \eqref{e01}, $T_\Lambda$ is injective. 
Set $e_\lambda=\{\|{\bf k}_\lambda\|\delta_{\lambda\mu}\}_{\mu\in\Lambda}$. Then $\|e_\lambda\|_{\ell^2(1/\|{\bf k}_\lambda\|^2)}=1$ and 
$R_\Lambda e_\lambda=E_\lambda\|{\bf k}_\lambda\|/E'(\lambda)$. 
Estimates \eqref{e14}, 
\eqref{e3}, \eqref{e5}, and \eqref{exy} show that the matrix $(\langle R_\Lambda e_\lambda,R_\Lambda e_\mu\rangle)_{\lambda,\mu\in\Lambda}$ consists of (i) a sequence of squares 
along the diagonal with uniformly bounded norms and (ii) rapidly decaying off-diagonal terms. 
Thus, $R_\Lambda$ is bounded from $\ell^2(1/\|{\bf k}_\lambda\|^2)$ to $\mathcal F_h$. 
By the argument in Subsection~\ref{ss11} 
we conclude that 
$\{\Bbbk_\lambda\}_{\lambda\in\Lambda}$ is a Riesz basis in $\mathcal F_h$.

Finally, for $n\ge 3$ we set 
$$
f_n(z)=\frac{E(z)}{z-R_n}\prod_{1\le s<n/2}\frac{z-e^{-2\pi i s/n}R_n}{z-e^{2\pi i s/n}R_n}.
$$
Then $|E_{R_n}|=|f_n|$ on the real line, and  
\begin{multline*}\label{e4}
\|f_n\|^2=\int_{\CC}|f_n(z)|^2e^{-h(z)}\,dm(z)\\ \ge 
\int_{\bigl||z|-R_n\bigr|\le\frac{R_n}{n},\,|\arg z-\pi/2|\le \pi/4}|f_n(z)|^2e^{-h(z)}\,dm(z)\\ \ge 
\frac{C}{nR_n}2^{n/20}\ge n\|E_{R_n}\|^2,\qquad n\to\infty.
\end{multline*}
This shows that $\mathcal F_h$ is not (up to the norm equivalence) isomorphic to a de Branges space.

\section{Proof of Proposition~\ref{ex}}
\label{pex}

Let $\mathcal F_h=\mathcal F_{h_0}$ (up to the norm equivalence), and let the polynomials belong to $\mathcal F_{h_0}$. We have 
$$
\|z^n\|^2_{\mathcal F_h}=\int_0^\infty 2\pi r^{2n+1}e^{-h(r)}\,dr=\int_0^\infty x^{n}u_h(x)\,dx,
$$
where $u_h(x)=\pi e^{-h(\sqrt{x})}$.

If $h$ is continuous and non-decreasing, we can find a non-increasing $\tilde u\in C^1(\mathbb R)$ such that $u_h\le \tilde u\le 2u_h$, $\tilde u'\le 0$.
Since
$$
\int_0^\infty x^{n}\tilde u(x)\,dx=-\frac1{n+1}\int_0^\infty x^{n+1}\tilde u'(x)\,dx,
$$
by the Cauchy--Schwarz inequality we obtain that for some bounded sequence $\{c_n\}_{n\ge 1}$,
\begin{equation}
\text{the sequence\quad}\bigl\{ \log[n\|z^n\|^2_{\mathcal F_{h_0}}]+c_n\bigr\}_{n\ge 1}\text{\quad is convex}.
\label{e19}
\end{equation}

On the other hand, let 
$$
p_n=\log \int_0^\infty x^{n}\,d\mu(x),
$$
where 
$$
d\mu(x)=\sum_{k\ge 1}e^{-s_kt_k}\delta_{\exp s_k},
$$
$t_k\in \mathbb N$, $s_k\ge 1$, $k\ge 0$. 
If $t_k>t_{k-1}+1$, $s_k>2t_ks_{k-1}$, $k>1$, then
$$
p_n=\log\sum_{k\ge 1}e^{(n-t_k)s_k}=\max_{k\ge 1}[(n-t_k)s_k]+O(1),\qquad n\ge 1,
$$
and
$$
p_n=(n-t_k)s_k+O(1) 
$$
for $t_k+1\le n\le t_{k+1}$, $k\ge 1$.

Finally, let $t_k\ge t^2_{k-1}$, $k>1$, let $u_{h_0}$ be a continuous positive function such that 
$$
\Bigl|p_n-\log \int_0^\infty x^{n}u_{h_0}(x)\,dx\Bigr|\le 1,\qquad n\ge 1,
$$
and let $u_{h_0}(x)=\pi e^{-h_0(\sqrt{x})}$. Then $h_0$ is continuous and there exists a bounded sequence $\{d_n\}_{n\ge 1}$ such that the sequence 
$$
\bigl\{ \log[\|z^n\|^2_{\mathcal F_{h_0}}]+d_n\bigr\}_{n\ge 1}
$$
is linear for $t_k+1\le n\le t_{k+1}$, $k\ge 1$, which contradicts to \eqref{e19}.


\begin{thebibliography}{99}

 
\bibitem{bbb}
A. Baranov, Y. Belov, A. Borichev, {\it Spectral synthesis in de Branges spaces}, Geometric and Functional Analysis {\bf 25} (2015), 417--452.
 

\bibitem{BMS} Yu. Belov, T. Mengestie, K. Seip,
{\it Discrete Hilbert transforms on sparse sequences},
Proc.\ London Math.\ Soc.\ {\bf 103} (2011), 73--105.

\bibitem{BDK} A.~Borichev, R.~Dhuez, K.~Kellay, {\it Sampling and interpolation in large Bergman and Fock spaces}, 
Journal of Functional Analysis  {\bf 242} (2007), 563--606.

\bibitem{bl} A. Borichev, Yu. Lyubarskii, {\it Riesz bases
of reproducing kernels in Fock type spaces},
J. Inst.\ Math.\ Jussieu {\bf 9} (2010), 449--461.

\bibitem{br} L. de Branges, {\it Hilbert Spaces of Entire
Functions}, Prentice--Hall, Englewood Cliffs, 1968.

\bibitem{Du} P. Duren, {\it Theory of $H^{p}$ spaces}, Pure and Applied Mathematics, Vol.\ 38,  Academic Press, New York--London, 1970.

\bibitem{ha} W. Hayman, {\it The local growth of power series: a survey of the Wiman--Valiron method}, Canad.\ Math.\ Bull.\ {\bf 17} (1974), 317--358. 

\bibitem{LM}  Yu.~Lyubarskii, E.~ Malinnikova, {\it On approximation of subharmonic functions},  J. Anal.\ Math.\  {\bf 83}  (2001), 121--149.

\bibitem{nik} N.~Nikolski, \textit{Operators, Functions, and Systems: an
Easy Reading. Vol. 2}, Math. Surveys Monogr., Vol. 93, AMS,
Providence, RI, 2002.

\bibitem{SW1} K.~Seip, {\it Density theorems for sampling and interpolation 
in the Bargmann--Fock space}, I, J. Reine Angew.\ Math.\ {\bf 429} (1992), 91--106.

\end{thebibliography}
\end{document}